\documentclass[draft]{amsart}

\usepackage{amssymb}
\usepackage{url}
\usepackage{amscd}
\usepackage{enumerate}
\usepackage[all]{xy}

\def \N{{\mathbb N}}

\def \R{{\mathbb R}}

\def \1{{\mathbb 1}}

\theoremstyle{plain}
\newtheorem{theorem}{Theorem}
\newtheorem{proposition}{Proposition}
\newtheorem{definition}{Definition}

\theoremstyle{remark}
\newtheorem{remark}{Remark}
\newtheorem{example}{Example}

\title[Asymmetric normed  Baire space. ]{Asymmetric normed  Baire space.}
\author{Mohammed Bachir}

\address{Laboratoire SAMM 4543, Universit\'e Paris 1 Panth\'eon-Sorbonne\\ Centre P.M.F. 90 rue Tolbiac\\
75634 Paris cedex 13\\
France}
\email{Mohammed.Bachir@univ-paris1.fr}
\begin{document}

\begin{abstract}
We prove that an asymmetric normed space  is never a Baire space  if the topology induced by the asymmetric norm is not equivalent to the topology of a norm. More precisely, we show that a biBanach asymmetric normed space  is a Baire space if and only if  it is isomorphic to its associated normed space.
\end{abstract}
\maketitle
\noindent {\bf 2010 Mathematics Subject Classification:} 54E52, 46S99, 46B20, 

\noindent {\bf Keyword, phrase:} Baire space, Asymmetric normed space, Quasi-metric space,  Index of symmetry.
\section{Introduction}
The question of knowing under what conditions  quasi-metric  spaces are Baire spaces, has been treated by several authors (the definition of quasi-metric space  will be reminded later). There exist some positive results in this direction which can be found for example in  \cite{Kl} and \cite{FG} (see also \cite{Cob}). However, it turns out that in the case of biBanach asymmetric normed spaces, the result is positive if and only if  the topology itself induces a topology of a norm. Indeed, the main result of this note is precisely to show  that if an asymmetric normed space is not isomorphic to its associated normed space, then it is not a Baire space (Theorem \ref{Baire}). Consequently, in asymmetric normed spaces, the conditions given in \cite[Theorem ~2.1]{FG} implicitly implies that these spaces are isomorphic to their associated normed spaces (in fact, we think that the condition of quasi-regularity in \cite{FG} implies that the index of symmetry introduced in \cite{Ba-Fl}, is strictly bigger than zero and so the asymmetric norm is equivalent to the associated norm). We do not know if the result of this paper extends to the general case of quasi-metric spaces (see Open quastion 1). However, we prove using our main result, that the semi-Lipschitz free space  over any quasi-metric space  is a Baire space if and only if the quasi-metric of the space  is equivalent to  its symmetrized metric. The concept of semi-Lipschitz free space over a quasi-metric space was introduced recently by A. Daniilidis, J. M. Sepulcre and F. Venegas in \cite{D-all}, in  the same way as the classical Lipschitz free space of Godefroy-Kalton in \cite{GK} (see also Section \ref{Index-Baire}). It follows easily from the results of this note that if the semi-Lipschitz-free  over a bicomplete quasi-metric space is a Baire space then it will be the same for the quasi-metric space. The converse of this fact remains an open question (see Open question 2, at the end of the paper). The tool used to establish our results is based on the use of the index of symmetry introduced recently in \cite{Ba-Fl}. 
\section{Notation and Definitions}
In this section, we will recall the classical notions that will be used subsequently. For literature on quasi-hemi-metric, asymmetric normed spaces, semi-Lipschitz functions  and their applications, we refer to \cite{Cob}, where the developement of Functional Analysis on these spaces is detailed. We refer also to the various studies developed in \cite{GRS, GRS1}, \cite{CobM}, \cite{AFG1, AFG2}, \cite{RS} and \cite{DJV}. Quasi-metric spaces and asymmetric norms have recently attracted a lot of interest in modern mathematics, they arise naturally when considering non-reversible Finsler manifolds \cite{CJ, DJV, PT}. For an introduction and recent study of asymmetric free spaces (or semi-Lipschitz free spaces), we refer to the recent paper \cite{D-all}.
\begin{definition} A quasi-hemi-metric space is a pair $(X, d)$, where $X \neq \emptyset$  and
$$d : X \times X \to  [0, +\infty)$$
is a function, called quasi-hemi-metric (or quasi-hemi-distance), satisfying:

$(i)$ $\forall x, y, z \in X: d(x, y) \leq  d(x, z) + d(z, y)$ (triangular inequality);

$(ii)$ $\forall x, y \in X: x = y  \iff  [d(x, y) = 0 \textnormal{ and } d(y,x)=0]$.

A quasi-metric space is a pair $(X, d)$ satisfying $(i)$ and the following condition $(ii')$: $\forall x, y \in X: x = y  \iff  d(x, y) = 0$.

\end{definition} 
We define respectively the open and closed balls  centered at $x$ with radious $r\geq 0$ as follows
$$B(x,r):=\lbrace y\in X: d(x,y)<r\rbrace.$$
$$B[x,r]:=\lbrace y\in X: d(x,y)\leq r\rbrace.$$
In a similar way, we define asymmetric norms on real linear spaces as follows.
\begin{definition} Let $X$ be a real linear space. We say that $\|\cdot|: X \to \R^+$ is an asymmetric norm on $X$
if  the following properties hold.
\begin{enumerate}[$(i)$]
    \item For every $\lambda \geq 0$ and every $x \in X$, $\|\lambda x|= \lambda \|x|$.
    \item For every $x, y \in X$, $\|x+y|\leq \|x|+\|y|$.
    \item For every $x \in X:  \|x|=\|-x|=0 \iff x=0$.
\end{enumerate}
\end{definition}
An asymmetric norm on a real linear space $(X,\|\cdot|)$ naturally induced a quasi-hemi-metric as follows : $$d(x,y)=\|y-x|, \hspace{3mm}  \forall x, y \in X.$$  
Notice that a sequence $(x_n)$ converges to $x$ if $d(x,x_n)=\|x_n-x|\to 0$ but in general we dont have that $\|x-x_n|\to 0$.

Every quasi-hemi-metric space $(X, d)$ has an associated metric space $(X,d_s)$ where $d_s(x,y):=\max \lbrace d(x,y),d(y,x) \rbrace$ for every $x,y \in X$. An asymmetric normed space $(X,\|\cdot|)$ has also an associated normed space $(X,\|\cdot\|_s)$, where $\|x\|_s:=\max(\|x|,\|-x|)$ for all $x\in X$.
\begin{definition} A quasi-hemi-metric space $(X, d)$ is said to be bicomplete if its associated metric space is complete. An asymmetric normed space is said to be a biBanach space if its associated normed space is a Banach space.
\end{definition}
\vskip5mm
\paragraph{\bf The space of continuous linear functionals.} Let $(X,\|\cdot|)$ be an asymmetric normed spaces  and $(\R,\|\cdot|_\R)$ be the asymmetric normed real line equipped with the asymmetric norm $\|t|_{\R}=\max\lbrace t,0\rbrace$. A linear functional $p: (X,\|\cdot|)\to (\R,\|\cdot|_{\R})$ is called bounded if there exists $C\geq 0$ such that $$p(x)\leq C\|x|, \hspace{2mm} \forall x\in X.$$
In this case, we denote $\|p|_{\flat}:=\sup_{\|x|\leq 1}p(x)$. We denote $X^{\flat}$ the convex cone of all bounded functionals  from $(X,\|\cdot|_X)$ into $ (\R,\|\cdot|_{\R})$.
It is known (see \cite[Proposition 2.1.2]{Cob}) that a linear functional $p$ is bounded if and only if it is continuous, which in turn is equivalent to being continuous at $0$. It is easy to see that $p\in X^{\flat}$, if and only if $p: (X,\|\cdot|)\to (\R,|\cdot|)$ is upper semicontinuous.   The constant $\|p|_{\flat}$ can be calculated also by the formula (see \cite[Proposition 2.1.3]{Cob})
$$\|p|_{\flat}=\sup_{\|x|= 1} \|p(x)|_{\R}.$$
 The topological dual of the associated normed space $X_s:=(X,\|\cdot\|_s)$ of $X$  is denoted $X^*$ and is equipped with the usual dual norm denoted $\|p\|_{*}=\sup_{\|x\|_s\leq 1} \langle p, x\rangle$, for all $p\in X^*$.  Note, from \cite[Theorem 2.2.2]{Cob}, that the convex cone $X^{\flat}$ is not trivial, that is, $X^{\flat}\neq \lbrace 0 \rbrace$ whenever $X\neq\lbrace 0 \rbrace$. We always have that
$$X^{\flat} \subset X^* \textnormal{ and }  \|p\|_{*}\leq \|p|_{\flat}, \textnormal{ for all } p\in X^{\flat}.$$ 
Thus, $X^{\flat}$ is a convex cone included in $(X^*,\|\cdot\|_{*})$ but is not a vector space in general. We say that $(X,\|\cdot|_X)$ and $(Y,\|\cdot|_Y)$ are isomorphic and we use the notation $(X,\|\cdot|_X) \simeq(Y,\|\cdot|_Y)$, if there exists a bijective linear operator $T: (X,\|\cdot|_X)\to (Y,\|\cdot|_Y)$ such that $T$ and $T^{-1}$ are bounded. 

\paragraph{\bf The space of semi-Lipschitz functions.} Let  $(X, d)$ be a quasi-hemi-metric space equipped with a distinguished point $x_0$ (called a base point of $X$) and $(\R,\|\cdot|_\R)$ be the asymmetric normed real line equipped with the asymetric norm $\|t|_{\R}=\max\lbrace t,0\rbrace$. A function $f : (X, d)\to (\R,\|\cdot|_{\R})$ is said to be semi-Lipschitz if, there exists a positive real number  $C_f\geq 0$ such that 
$$\|f(x)-f(y)|_{\R}=\max \lbrace f(x)-f(y), 0\rbrace \leq C_f d(x,y), \hspace{3mm} \forall x, y\in X.$$
In general $-f$ is not semi-Lipschitz function, when $f$ is.  The convex cone of all semi-Lipschitz functions that vanish at $x_0$ will be denoted by $\textnormal{SLip}_0(X)$. 
We denote $$\|f|_L:=\sup_{d(x,y)>0 } \frac{\|f(x)-f(y)|_{\R}}{d(x,y)}; \hspace{3mm} \forall f\in \textnormal{SLip}_0(X),$$

\section{The main result} \label{Index-Baire}

 We define the index of symmetry of a  quasi-hemi-metric space $(X, d)$ in the same way as that introduced in \cite{Ba-Fl}  for asymmetric normed spaces:  
$$c(X):=\inf_{d(y,x)>0} \frac{d(x,y)}{d(y,x)} \in [0,1].$$
The conjugate of $c(X)$ is defined by $\overline{c}(X):=\sup_{d(y,x)>0} \frac{d(x,y)}{d(y,x)}$.
In the case of asymmetric normed space, using the quasi-metric $d(x,y)=\|y-x|$, we see that $c(X)=\inf_{\|x|=1} \|-x|$.

This index measures the degree of symmetry of the quasi-hemi-metric $d$, it was firstly introduced and studied in the assymetric normed space in \cite{Ba-Fl}, where a topological classification of asymmetric normed spaces has been given according to this index. Similarly to the framework of asymmetric normes spaces, we have the following proposition.

\begin{proposition} \label{index} Let $(X, d)$ be a quasi-hemi-metric space. Then, the following assertions hold.

$(i)$ $c(X) \in [0,1]$.

$(ii)$ If $c(X)>0$, then $\overline{c}(X)=\frac{1}{c(X)}\in[1,+\infty[$. If  moreover, $(X, d)$ is a quasi-metric space (that is $d(x,y)=0$ iff $x=y$) then, this formula extends to the case when $c(X)=0$, with $\overline{c}(X)=+\infty$. 

$(iii)$ $d$ is a metric if and only if $c(X)=1$.

\end{proposition}
\begin{proof} $(i)$ It is clear that $c(X)\geq 0$. Suppose by contradiction that $c(X)>1$. Then, for every $x, y\in X$ such that $d(y,x)>0$ we have that $d(x,y)>d(y,x)>0$. Thus, since $d(x,y)>0$, we also have that $d(y,x)>d(x,y)$ which is a contradiction.  Thus,  $c(X) \in [0,1]$. To see $(ii)$, we observe by definition that $c(X)d(y,x)\leq d(x,y)$ for all $x, y\in X$ (even if $c(X)=0$). If, $c(X)>0$, then we get that $\overline{c}(X)\leq \frac{1}{c(X)}<+\infty$. On the other hand, we have that $d(x,y)\leq \overline{c}(X)d(y,x)$ for all $x,y \in X$. This implies that $\frac{1}{\overline{c}(X)}\leq c(X)$. Hence, we have that $\overline{c}(X)=\frac{1}{c(X)}\in[1,+\infty[$. If $c(X)=0$, then, there exists a pair of sequences $(a_n,b_n)\in X\times X$ such that $d(b_n,a_n)>0$ and $\frac{d(a_n,b_n)}{d(b_n,a_n)}\to 0$. Since $(X, d)$ is a quasi-metric space and $d(b_n,a_n)>0$, then $d(a_n,b_n)>0$. Thus, $\frac{d(b_n,a_n)}{d(a_n,b_n)}\to +\infty$ and so $\overline{c}(X)=+\infty=\frac{1}{c(X)}$. For part $(iii)$, suppose that $c(X)=1$. In this case, we have that $d(x,y)\geq d(y,x)$ for all $x,y\in X$, which implies that $d(x,y)=d(y,x)$ for all $x,y\in X$. Hence $d$ is a metric. The converse is trivial.
\end{proof}
Recently, we proved in \cite{Ba-Fl}, the following result that we will use. The difficult part in this theorem is $(iii) \Longrightarrow (i)$. 

\begin{theorem} \label{Ba-Fl} \textnormal{(\cite[Corollary 3]{Ba-Fl})} Let $(X, \|\cdot|)$ be an asymmetric normed space. The following assertions are equivalent. 

$(i)$ $c(X)>0$.

$(ii)$ $ (X, \|\cdot|) $ is isomorphic to its associated normed space.

$(iii)$ $X^{\flat}$ is a vector space.

\end{theorem}

\vskip5mm

Now, we give our main result.

\begin{theorem} \label{Baire}  Let $(X,\|\cdot|)$ be an asymmetric normed  space. Suppose that $c(X)=0$. Then, $(X,\|\cdot|)$ is not a Baire space. If moreover, we assume that $(X,\|\cdot|)$ is  biBanach, then the converse is also. Consequently, a biBanach asymmetric normed space is a Baire space if and only if it is isomorphic to its associated normed space. 
\end{theorem}

\begin{proof} Suppose that $c(X)=0$, it follows using Theorem \ref{Ba-Fl}, that $X^{\flat}$ is not a vector space. Thus, there exists a linear functional $p\in  X^{\flat}$ such that $-p\not \in X^{\flat}$. Since $p\in X^{\flat}$, then $p$ is upper semicontinuous from $(X,\|\cdot|)$ into $(\R,|\cdot|)$. For each $n\in \N$, let $F_n:=\lbrace x \in X: -p(x)\leq n \rbrace$. Then $F_n$ is a closed subset of $(X,\|\cdot|)$ (since $-p$ is lower semicontinuous) and we have that $X=\cup_{n\in \N}F_n$. Suppose by contradiction that $(X,\|\cdot|)$ is a Baire space. Then, there exists $n_0\in \N$ such that  $F_{n_0}$ has a non empty interior. Let $\varepsilon >0$ and $x_0\in X$ such that  $B_{\|\cdot|}(x_0,\varepsilon) \subset F_{n_0}$. For each $h\in B_{\|\cdot|}(0,1)$, we have that $x_0+\varepsilon h\in B_{\|\cdot|}(x_0,\varepsilon)$. Thus, $$-p(x_0+\varepsilon h)\leq n_0.$$ 
In other words, for each $h\in B_{\|\cdot|}(0,1)$ we have that $-p(h)\leq \frac{n_0+p(x_0)}{\varepsilon}$. This implies that $\|-p|_{\flat}\leq \frac{n_0+p(x_0)}{\varepsilon}<+\infty$  which contradict the fact that $-p\not \in X^{\flat}$.

Now, to see the second part, if $c(X)>0$, then we have that (see \cite[Proposition 3.]{Ba-Fl}):
\begin{eqnarray*} \label{eq2}
c(X)\|x\|_s\leq \|x|\leq \|x\|_s, \hspace{2mm} \forall x\in X.
\end{eqnarray*}
 Thus, as $(X, \|\cdot\|_s)$ is a Banach space, it is a Baire space and so  $(X, \|\cdot|)$ is also a Baire space, since its topology is equivalent to the norm topology by the above inequality. The equivalence between the condition $c(X)>0$ and the fact that $(X,\|\cdot|)$ is isomorphic to its associated normed space, is trivial from the above inequalities.
\end{proof}
\begin{example} \label{ex}   Let $X= l^{\infty}(\N^*)$ equipped with the asymmetric norm $\|\cdot|_{\infty}$ defined by 
$$\|x|_{\infty}=\sup_{n\in \N^*}\|x_n|_{\frac 1 n}\leq \|x\|_{\infty},$$
where for each $t\in \R$ and each $n\in \N^*$, $\|t|_{\frac 1 n}=t$  if $t\geq 0$ and $\|t|_{\frac 1 n}=-\frac{t}{n}$ if $t\leq 0$. Then, for each $n\in \N^*$, we have $\|e_n|_{\infty}=1$ and $\|-e_n|_{\infty}=\frac 1 n$, where $(e_n)$ is the canonical basis of $c_0(\N^*)$. It follows that $c(l^{\infty}(\N^*))=0$ and so by the above theorem, $( l^{\infty}(\N^*), \|\cdot|_{\infty})$ is not a Baire space.
\end{example}
\begin{remark} The above theorem extends easily to asymmetric normed cone with the same proof, replacing "biBanach" by "bicomplete" (we refer to \cite{D-all, Cob} for more informations about asymmetric normed cones and their duals).
\end{remark}
Notice that it is easy to see that for any quasi-hemi-metric space $(X,d)$, we have that $ c(X)d_s\leq d\leq d_s$ ($c(X)$ my be equal to zero). Thus,  if $c(X)>0$ and $(X,d)$ is bicomplete (equivalently $(X,d_s)$ is a complete metric space), then $(X,d)$ is a Baire space as $d$ is equivalent to the complete metric $d_s$. We dont know if the converse remains true. We have the following open question.
\vskip5mm
\noindent{\bf Open question 1.} Does the result of Theorem \ref{Baire} extends to any quasi-hemi-metric space? More precisely, if $(X,d)$ is any quasi-hemi-metric space, does $c(X)=0$ implies that $(X,d)$ is not a Baire space?
\vskip5mm
We are going to prove in Theorem \ref{A-free} bellow, that the condition $c(X)=0$ in quasi-hemi-metric space,  is however always equivalent to the fact that the semi-Lipschitz free space $\mathcal{F}_a(X)$ over $(X, d)$ (introduced recently in \cite{D-all}) is not a Baire space. We recall below the notion of semi-Lipschitz free space and we refer to \cite{D-all} for more details.
\vskip5mm
 \paragraph{\bf The semi-Lipschitz free space.}  The semi-Lipschitz free space over a quasi-hemi-metric space $X_d:=(X,d)$ was recently introduced and studied in \cite{D-all}. Its construction is analogous to the classical Lipschitz free space over a metric space introduced by Godefroy-Kalton in \cite{GK}.  Notice that the cone $\textnormal{SLip}_0 (X_d)$ in our paper corresponds to the cone $\textnormal{SLip}_0 (X_{\overline{d}})$ in the paper \cite{D-all}, where $\overline{d}$ denotes the conjugate of $d$ defined by $\overline{d}(x,y)=d(y,x)$ for all $x, y\in X$.  For every $x\in X$ we consider the corresponding evaluation mapping
\begin{eqnarray*}
\delta_x : \textnormal{SLip}_0(X_{\overline{d}})&\to& \R\\ 
                      f &\mapsto& f(x)
\end{eqnarray*}
It is shown in \cite[Proposition 3.1]{D-all} that for each $x \in X$, both the evaluation
functional $\delta_x$ and  its opposite $-\delta_x$ belong to the dual cone of $\textnormal{SLip}_0(X_{\overline{d}})$. The semi-Lipschitz free space over $(X, d)$, denoted by $\mathcal{F}_a(X)$, is defined as the (unique) bicompletion of the asymmetric normed space $(\textnormal{span} \lbrace\delta_x : x\in X \rbrace, \|\cdot|^*)$, where $\|\cdot|^*$ is the restriction of the norm of the dual cone of $\textnormal{SLip}_0(X_{\overline{d}})$ (corresponding to $\|\cdot|_{\flat}$ in our notation). It is proved in \cite[Theorem 3.5]{D-all} that  the
dual cone of the asymmetric normed space $\mathcal{F}_a(X)$ (that is $\mathcal{F}_a(X)^{\flat}$, also denoted by $\mathcal{F}_a(X)^{*}$  in \cite{D-all})  is isometrically isomorphic to $\textnormal{SLip}_0(X_{\overline{d}})$ and  that (see \cite[Proposition 3.3]{D-all})
\begin{eqnarray*}
\delta : (X,d) &\to& (\mathcal{F}_a(X), \|\cdot|^*)\\ 
                      x &\mapsto& \delta_x
\end{eqnarray*}
is an isometry onto its image.  Therefore, for any $x, y \in X$, we have
$$\overline{d}(x,y)=d(y,x)=\|\delta_x-\delta_{y}|^*.$$
\begin{proposition} \label{Comp-index} Let $(X,d)$ be  a quasi-hemi-metric space and $\mathcal{F}_a(X)$ be the semi-Lipschitz free space over $(X, d)$. Then, we have that $c(\mathcal{F}_a(X))\leq c(X)$. Moreover,  we have that 

$(i)$ $c(X)=0$ if and only if $c(\mathcal{F}_a(X))=0$.

$(ii)$ $c(X)=1$ if and only if $c(\mathcal{F}_a(X))=1$.
\end{proposition}
\begin{proof} By the definition of the index of symmetry, we have that 
\begin{eqnarray*}
c(\mathcal{F}_a(X)) &:=&\inf_{Q\in \mathcal{F}_a(X); \|Q|^*>0}\frac{\|-Q|^*}{\|Q|^*}\\
                                 &\leq& \inf_{x, y\in X; \|\delta_x -\delta_{y}|^*>0}\frac{\|\delta_{y} -\delta_{x}|^*}{\|\delta_x -\delta_{y}|^*}\\
                                 &=& \inf_{d(y,x)>0}\frac{d(x,y)}{d(y,x)}\\
                                 &=& c(X).
\end{eqnarray*}
Clearly $c(X)=0$ implies that $c(\mathcal{F}_a(X))=0$. To see the converse,  suppose that $c(\mathcal{F}_a(X))=0$. By applying  Theorem \ref{Ba-Fl} to the asymmetric normed space $\mathcal{F}_a(X)$, we get that $\mathcal{F}_a(X)^{\flat}$ is not a vector space. Thus, there exists $f\in \mathcal{F}_a(X)^{\flat}=\textnormal{SLip}_0(X_{\overline{d}})$ such that $-f\not \in \mathcal{F}_a(X)^{\flat}=\textnormal{SLip}_0(X_{\overline{d}})$. Thus, there exists a pair of sequences $((a_n),(b_n))\subset X\times X$ such that for all $n\in \N\setminus \lbrace 0 \rbrace$, we have
\begin{eqnarray*}
nd(a_n,b_n) < f(a_n)-f(b_n) \leq \|f|_Ld(b_n,a_n).
\end{eqnarray*}
This implies that:  $\forall n\in \N\setminus \lbrace 0 \rbrace$, $d(b_n,a_n)>0$ and $\frac{d(a_n,b_n)}{d(b_n,a_n)}<\frac{\|f|_L}{n}\to0$. This shows that $c(X)=0$.

For the second part,  if  $c(\mathcal{F}_a(X))=1$, then clearly $c(X)=1$ since the index of symmetry belongs to the segment $[0,1]$. Now, suppose that $c(X)=1$, then by part $(iii)$ of Proposition \ref{index}, we have that $(X,d)$ is a metrci space. It follows that $\textnormal{SLip}_0(X_{\overline{d}})=\textnormal{Lip}_0(X_{d})$ is the classical Banach space of real-valued Lipschitz functions on $(X,d)$ that vanish at $x_0$ and so $\mathcal{F}_a(X)$ coincides with the classical Lipschitz-free space which is a Banach space. Thus $c(\mathcal{F}_a(X))=1$ by part $(iii)$ of Proposition \ref{index}.
\end{proof}
We dont know if $c(\mathcal{F}_a(X))= c(X)$, but it probably seems true.
\begin{theorem} \label{A-free} Let $(X,d)$ be  a quasi-hemi-metric space and $\mathcal{F}_a(X)$ be the semi-Lipschitz free space over $(X, d)$.  Then, $c(X)=0$ if and only if $(\mathcal{F}_a(X),\|\cdot|^*)$ is not a Baire space.
\end{theorem}
\begin{proof} The proof follows from Proposition \ref{Comp-index} and Theorem \ref{Baire}, noting that  $(\mathcal{F}_a(X),\|\cdot|^*)$ is a biBanach space by construction.
\end{proof}
From the above result, we can reformulate the remark just before the Open question 1, as follows.

\begin{proposition} Suppose that $(X,d)$ is a bicomplete quasi-hemi-metric space such that  $(\mathcal{F}_a(X),\|\cdot|^*)$  is a Baire space. Then $(X,d)$ is also a Baire space.
\end{proposition} 
We dont know if the converse in the above proposition is true in general. 
\vskip5mm
\noindent{\bf Open question 2.} Let $(X,d)$ be a quasi-hemi-metric space. Suppose that $(X,d)$ is a Baire space, does $(\mathcal{F}_a(X),\|\cdot|^*)$  is also a Baire space?

\vskip5mm
It  is clear that Open question 1 and Open question 2, are in fact equivalent.
\vskip5mm
\section*{Declaration} The author declare that there is no conflict of interest.
\section*{Acknowledgment}
This research has been conducted within the FP2M federation (CNRS FR 2036) and  SAMM Laboratory of the University Paris Panthéon-Sorbonne. 

\bibliographystyle{amsplain}

\end{document}